\documentclass{article}
\usepackage{amsfonts}

\newtheorem{theorem}{Theorem}

\newtheorem{corollary}[theorem]{Corollary}

\newtheorem{definition}[theorem]{Definition}
\newtheorem{example}[theorem]{Example}

\newtheorem{lemma}[theorem]{Lemma}

\newtheorem{proposition}[theorem]{Proposition}
\newtheorem{remark}[theorem]{Remark}

\newenvironment{proof}[1][Proof]{\noindent\textbf{#1.} }{\ \rule{0.5em}{0.5em}}
\begin{document}

\begin{center}
\textbf{Accumulation points of the sets of real parts of zeros of the
partial sums of the Riemann zeta function }

\textbf{G.\ Mora}

\textit{Department of Mathematical Analysis. University of Alicante. 03080
Alicante (Spain)}

\textit{gaspar.mora@ua.es}
\end{center}

\textbf{ABSTRACT }Let $\zeta _{n}(z):=\sum_{k=1}^{n}\frac{1}{k^{z}}$ be the 
$nth$ partial sum of the Riemann zeta function. In this paper it is shown
that for every integer $n>2$ there exists $\delta _{n}>0$ such that $%
R^{(n)}:=\overline{\left\{ \Re z:\zeta _{n}(z)=0\right\} }$ contains to the
interval $\left[ -\delta _{n},b^{(n)}\right] $, where $b^{(n)}:=\sup \left\{
\Re z:\zeta _{n}(z)=0\right\} $. It is also demonstrated that $b^{(n)}$ is
positive for all $n>2$. Then, for every $n>2$, $\zeta _{n}(z)$ possesses
infinitely many zeros, with negative and positive real part, arbitrarily
close to any line contained in the strip $\left[ -\delta _{n},b^{(n)}\right]
\times 
%TCIMACRO{\U{211d} }%
%BeginExpansion
\mathbb{R}
%EndExpansion
$. Finally, noting $b^{(n)}=1+\left( \frac{4}{\pi }-1+o(1)\right) \frac{%
\log \log n}{\log n}$, $n\rightarrow \infty $, it has been deduced that $%
\cap _{n=3}^{\infty }R^{(n)}$ contains a non-degenerated interval containing 
$0$, so $0$ is an accumulation point common to all the sets $P_{\zeta
_{n}}:=\left\{ \Re z:\zeta _{n}(z)=0\right\} $, $n>2$, and in particular $%
0\in R^{(n)}$ for all $n\geq 2$.



\section{Introduction}

The distribution of the zeros of exponential polynomials\ was mainly studied
in the first third of $XX^{th}$ century linked to the development of the
theory of differential equations, as we can see for instance in $\left[ 8%
\mbox{, }23\mbox{, }24\mbox{, }28\right] $. 
On the concrete case of the
partial sums of the Riemann zeta function, $\zeta _{n}(z):=\sum_{k=1}^{n}%
\frac{1}{k^{z}}$, the interest for knowing the distribution of their zeros
had its starting point in 1948 as consequence of a paper of Tur\'{a}n $\left[
26\right] $ where he pointed out an intriguing connexion betweeen a
particular distribution of such zeros near the line $x=1$ and the Riemann
Hypothesis (briefly R.H.). Namely, Tur\'{a}n 
$\left[ 26\mbox{, Th. I}\right]$
 proved that if 
\[
\zeta _{n}(z)\neq 0\mbox{ for }\Re z\geq 1+n^{-\frac{1}{2}+\epsilon }\mbox{, 
}n\geq N_{0}(\epsilon )\mbox{,} 
\]%
then the Riemann Hypothesis is true.

The first step to avoid Tur\'{a}n's dream of proving R.H. via the previous
theorem was given by Haselgrove $\left[ 7\right] $ few years later. Indeed,
by means of an added note at the end of $\left[ 7\right] $, Haselgrove
showed that the suggestion of Tur\'{a}n\ $\left[ 26\mbox{, (25.1)}\right] $
on the non-negativity for all $x\geq 1$ of the function $T(x):=\sum_{k\leq x}%
\frac{\lambda (k)}{k}$, was false, where $\lambda (k):=(-1)^{p(k)}$ and $%
p(k)$ is the number of prime factors of $k$ counting multiplicities. Hence,
for some $n$, one has 
\begin{equation}\label{1.1}
\sum_{k=1}^{n}\frac{\lambda (k)}{k}<0\mbox{.}  %\tag{1.1}
\end{equation}%
This $n$, allows to define the Dirichlet polynomial 
\[
D_{n}(z):=\sum_{k=1}^{n}\frac{\lambda (k)}{k^{z}} 
\]%
and then the corresponding partial sum 
\[
\zeta _{n}(z)=\sum_{k=1}^{n}\frac{1}{k^{z}}\mbox{,} 
\]%
because the equivalence theorem of Bohr $\left[ 4\right] $, attains the same
set of values in any half-plane $\Re z>a$. That is, for arbitrary real $%
a$, 
\[
\left\{ D_{n}(z):\Re z>a\right\} =\left\{ \zeta _{n}(z):\Re %
z>a\right\} \mbox{.} 
\]%
Thus, since for sufficiently large real $z$, $D_{n}(z)$ is near $1$ and, by
(1.1), $D_{n}(1)<0$, there exists a real zero $x>1$ of $D_{n}(z)$.
Therefore, there is a zero of $\zeta _{n}(z)$ in the half-plane $\Re z>1$.

In 1968 Spira $\left[ 22\right] $ demonstrated the same for $n=19$. That is,
the partial sum $\zeta _{19}(z)$ has a zero in the half-plane $\Re z>1$%
. Levinson $\left[ 9\mbox{, Theorem 1}\right] $ in 1973 found an asymptotic
formula, for large $n$, for the location of the zeros of $\zeta _{n}(z)$
near the point $z=1$. In particular, he proved that theses zeros have real
part less than $1$. Voronin $\left[ 27\right] $ in 1974 showed that $\zeta
_{n}(z)$ has zeros in $\Re z>1$ for infinitely many $n$.

A sharp result on the upper bound 
\[
b^{(n)}:=\sup \left\{ \Re z:\zeta _{n}(z)=0\right\} 
\]%
was demonstrated in 2001 by Montgomery and Vaughan in $\left[ 12\right] $.
Namely, they proved that there exists $N_{0}$ such that if $n>N_{0}$ then $%
\zeta _{n}(z)\neq 0$ whenever 
\begin{equation}\label{1.2}
\Re z\geq 1+\left( \frac{4}{\pi }-1\right) \frac{\log \log n}{\log n}\mbox{%
.}  %\tag{1.2}
\end{equation}%
But in the opposite direction, in 1983, Montgomery $\left[ 11\right] $ had
already shown that\ for every $0<c<\frac{4}{\pi }-1$ there exists $N_{0}(c)$
such that $\zeta _{n}(z)$ has zeros in the half-plane 
\begin{equation}\label{1.3}
\Re z>1+c\frac{\log \log n}{\log n}\mbox{, }n>N_{0}(c)\mbox{.}  %\tag{1.3}
\end{equation}%
Since\ for any positive $c$, $n$ sufficiently large and any $0<\epsilon <%
\frac{1}{2}$, one has 
\[
c\frac{\log \log n}{\log n}>n^{-\frac{1}{2}+\epsilon }\mbox{,} 
\]%
Montgomery's result implies the existence of zeros of $\zeta _{n}(z)$ having
real part much larger than the bound $1+n^{-\frac{1}{2}+\epsilon }$
considered by Tur\'{a}n in $\left[ 26\mbox{, Th. I}\right] $. Hence, this
old theorem of Tur\'{a}n of 1948 became vacuous and consequently the hope of
proving R.H. via such hypothetical localization of the zeros of the partial
sums of the Riemann zeta function died.

The results of Montgomery and Vaughan, (1.2) and (1.3), imply that, fixed $%
0<c<\frac{4}{\pi }-1$, for all $n>N_{1}:=\max \left\{
N_{0},N_{0}(c)\right\} $, the upper bound $b^{(n)}$ satisfies the
inequalities 
\begin{equation}\label{1.4}
1+c\frac{\log \log n}{\log n}<b^{(n)}\leq 1+\left( \frac{4}{\pi }-1\right) 
\frac{\log \log n}{\log n}\mbox{.}  %\tag{1.4}
\end{equation}%
Then, since $\lim_{n\rightarrow \infty }\frac{\log \log n}{\log n}=0$, it
follows that 
\[
\lim_{n\rightarrow \infty }b^{(n)}=1\mbox{.} 
\]%
On the other hand, Balazard and Vel\'{a}squez-Casta\~{n}\'{o}n $\left[ 2%
\right] $ demonstrated that the lower bound 
\[
a^{(n)}:=\inf \left\{ \Re z:\zeta _{n}(z)=0\right\} 
\]%
is so that 
\begin{equation}\label{1.5}
\lim_{n\rightarrow \infty }\frac{a^{(n)}}{n}=-\log 2\mbox{.}  %\tag{1.5}
\end{equation}%
Therefore 
\[
\lim_{n\rightarrow \infty }a^{(n)}=-\infty \mbox{.} 
\]

Once the estimation of the bounds $a^{(n)}$ and $b^{(n)}$ has been settled,
our interest is focused on the distribution of the real parts of the zeros
of\ $\zeta _{n}(z)$ in the interior of its critical interval $\left[
a^{(n)},b^{(n)}\right] $ and, in this sense, we suggest to see $\left[ 6%
\mbox{, }13\mbox{, }14\mbox{, }15\mbox{, }16\right] $. A way of knowing the
distribution of the zeros of $\zeta _{n}(z)$ is by means of the study of the
set of their real projections 
\[
P_{\zeta _{n}}:=\left\{ \Re z:\zeta _{n}(z)=0\right\} 
\]%
and its closure set, denoted by $R^{(n)}$. Indeed, the difference between $%
R^{(n)}$ and the critical interval $\left[ a^{(n)},b^{(n)}\right] $ gives us
a certain measure of the uniformity of the distribution of the real part of
the zeros of $\zeta _{n}(z)$ and, consequently, the position of those zeros
with respect to vertical lines contained in the critical strip $\left[
a^{(n)},b^{(n)}\right] \times 
%TCIMACRO{\U{211d} }%
%BeginExpansion
\mathbb{R}
%EndExpansion
$ where are located all the zeros of $\zeta _{n}(z)$.

The zeros of $\zeta _{2}(z)=1+\frac{1}{2^{z}}$, or those of the function $%
G_{2}(z):=$ $1+2^{z}$, are of the form 
\begin{equation}\label{1.6}
\frac{\pi (2k+1)i}{\log 2}\mbox{, }k\in 
%TCIMACRO{\U{2124} }%
%BeginExpansion
\mathbb{Z}
%EndExpansion
\mbox{.}  %\tag{1.6}
\end{equation}%
Then $a^{(2)}=b^{(2)}=0$, so $R^{(2)}=\left\{ 0\right\} $ and trivially $%
R^{(n)}=$ $\left[ a^{(n)},b^{(n)}\right] $ for $n=2$.

Since $R^{(n)}\subset $ $\left[ a^{(n)},b^{(n)}\right] $ for all $n\geq 2$,
the best possible relation between both sets one reaches when they are
equal. Thus, whenever $R^{(n)}=$ $\left[ a^{(n)},b^{(n)}\right] $, we will
say that the distribution of the real parts of the zeros of $\zeta _{n}(z)$
is completely uniform on its critical interval. It means that, given any
vertical line contained in $\left[ a^{(n)},b^{(n)}\right] \times 
%TCIMACRO{\U{211d} }%
%BeginExpansion
\mathbb{R}
%EndExpansion
$, $\zeta _{n}(z)$ possesses infinitely many zeros arbitrarily close to that
line. We have proved $\left[ 15\mbox{, Th.12}\right] $ that this is exactly
the case of the partial sums $\zeta _{n}(z)$ for sufficiently large $n$.
That is, we have demonstrated in $\left[ 15\mbox{, Th.12}\right] $ the
existence of a positive integer $N$ such that 
\begin{equation}\label{1.7}
R^{(n)}=\left[ a^{(n)},b^{(n)}\right] \mbox{ for all }n>N\mbox{.}  %\tag{1.7}
\end{equation}%
This result points out that the distribution of the real part of the zeros
of $\zeta _{n}(z)$, for all $n>N$, is completely uniform on its
corresponding critical interval. Consequently, the partial sums $\zeta
_{n}(z)$, $n>N$, in spite of the set of their frequencies $\left\{ -\log
k:k=2,...,n\right\} $ is not linearly independent (briefly l.i.) over the
rationals for any $n>3$, share with the Dirichlet polynomials, with l.i.
frequencies and coefficients of modulus $1$, the property of having the real
parts of their zeros uniformly distributed on their critical intervals (see $%
\left[ 16\mbox{, Th. 10}\right] $).

In the present paper, by means of our Theorem 2, we have given a positive
answer to an old question, on the negativity of the real part of the zeros
of the partial sums of the Riemann zeta function, posed by Bellman and Hooke
in 1963 $\left[ 3,\mbox{ p. 439, Question 25}\right] $. It has been also
proved in our Theorem 17 that for every integer $n>2$ there exists a number $%
\delta _{n}>0$ such that $\left[ -\delta _{n},b^{(n)}\right] \subset R^{(n)}$%
. Then, whenever $n>N$, by (1.7), the Theorem 17 is improved of an optimal
form because the $\delta _{n}$'s can be taken equal to $-a^{(n)}=n\log 2+o(n)
$, noting (1.5). Some relevant consequences are derived from Theorem 17:

a) By (1.4) and Corollary 3 of this paper, the set $\cap _{n=3}^{\infty
}R^{(n)}$ contains an interval of the form $\left[ -\delta ,b\right] $ for
some $\delta \geq 0$ and $b:=\min \left\{ b^{(n)}:n>2\right\} $, which is a
positive number. Then\ the strip $\left[ -\delta ,b\right] \times 
%TCIMACRO{\U{211d} }%
%BeginExpansion
\mathbb{R}
%EndExpansion
$ is common to all the $nth$ partial sums $\zeta _{n}(z)$, $n>2$. That is,
given any vertical line contained in $\left[ -\delta ,b\right] \times 
%TCIMACRO{\U{211d} }%
%BeginExpansion
\mathbb{R}
%EndExpansion
$, there are infinitely many zeros of every $\zeta _{n}(z)$, for all $n>2$,
arbitrarily close to that line.

b) In particular, we deduce that $0$ is an accumulation point common to all
the sets $P_{\zeta _{n}}=\left\{ \Re z:\zeta _{n}(z)=0\right\} $, $n>2$.
Moreover, because of Lemma 1, for every $n>2$, there exist infinitely many
zeros of each $\zeta _{n}(z)$, with negative and positive real parts, which
are arbitrarily close to the imaginay axis. Hence, the Theorem 17 and its
consequences have applications to the stability of the solutions of neutral
functional differential equations with delays (see for instance $\left[ 10%
\right] $).

c) By regarding the partial sum of Riemann zeta function $\zeta _{2}(z):=1+%
\frac{1}{2^{z}}$, noting (1.6), one has 
\[
\cap _{n=2}^{\infty }R^{(n)}=\left\{ 0\right\} \mbox{,} 
\]%
which proves Sepulcre and Vidal's conjecture $\left[ 20\mbox{, Conjecture 15}%
\right] $.

\section{Distribution of the zeros of $\protect\zeta _{n}(z)$ on the
half-planes $\Re z>0$ and $\Re z<0$}

\begin{lemma}
Let $f(z)$ be an almost-periodic function not identically null on a strip $%
S_{a,b}:=\left\{ z:a<\Re z<b\right\} $. If $f(z)$ has a zero at a point $%
z_{0}\in S_{a,b}$ then it has infinitely many zeros in $S_{\epsilon
}:=\left\{ z:\Re z_{0}-\epsilon <\Re z<\Re z_{0}+\epsilon \right\} $ for
arbitrary $\epsilon >0$.
\end{lemma}

\begin{proof}
We firstly note that, since the almost-periodicity implies analyticity $%
\left[ 5\mbox{, }p.101\right] $, $f(z)$ is analytic on $S_{a,b}$. Secondly,
by $\left[ 5\mbox{, }p.102\right] $, $f(z)$ is bounded on any closed
sub-strip $S_{\alpha ,\beta }:$ $=\left\{ z:a<\alpha \leq \Re z\leq \beta
<b\right\} $ of $S_{a,b}$. Now, let $z_{1}$ be a point on the line $x=\Re
z_{0}$ such that $f(z_{1})\neq 0$. Thus we claim that there exist positive
numbers $\delta $ and $l$ such that on any segment of length $l$ of the line 
$x=\Re z_{0}$ there is a point $z_{1}+iT$ verifying 
\begin{equation}\label{2.1}
\left\vert f(z_{1}+iT)\right\vert \geq \delta \mbox{.}  %\tag{2.1}
\end{equation}%
Indeed, by taking $\delta =\frac{\left\vert f(z_{1})\right\vert }{2}>0$,
the almost-periodicity of $f(z)$ involves the existence of a real number $%
l=l(\delta )$\ such that every interval of length $l$\ on the imaginary axis
contains at least a number $T_{1}=T_{1}(\delta )$, called translation
number,\ satisfying 
\begin{equation}\label{2.2}
\left\vert f(z+iT_{1})-f(z)\right\vert \leq \delta \mbox{ for all }z\in
S_{a,b}\mbox{.}  %\tag{2.2}
\end{equation}%
Then, by setting $z=z_{1}$ in (2.2) and according to the choice of $\delta $%
, the inequality (2.1) follows, as claimed. By reiterating this process for
every $\delta _{n}=\frac{\delta }{2^{n-1}}$, $n\geq 1$, the
almost-periodicity of $f(z)$ assures the existence of a real number $%
l_{n}=l_{n}(\delta _{n})$ such that every interval of length $l_{n}$
contains at least a translation number $T_{n}=T_{n}(\delta _{n})$ which,
without loss of generality, can be chosen in such a way that $%
(T_{n})_{n=1,2,..}$ be an increasing sequence of positive numbers with $%
T_{n+1}-T_{n}>1$, for all $n$, satisfying 
\begin{equation}\label{2.3}
\left\vert f(z+iT_{n})-f(z)\right\vert \leq \delta _{n}\mbox{ for all }z\in
S_{a,b}\mbox{.}  %\tag{2.3}
\end{equation}%
Then, by just making $z=z_{0}$ in (2.3), we have 
\begin{equation}\label{2.4}
\lim_{n\rightarrow \infty }f(z_{0}+iT_{n})=0\mbox{.}  %\tag{2.4}
\end{equation}%
We now consider the sequence of analytic functions on $S_{a,b}$ 
\begin{equation}\label{2.5}
g_{n}(z):=f(z+iT_{n})\mbox{, }z\in S_{a,b}\mbox{, }n=1,2,...\mbox{, } 
%\tag{2.5}
\end{equation}%
which are uniformly bounded on the compacts of $S_{a,b}$ by virtue of the
boundedness of $f(z)$ on any substrip of $S_{a,b}$. Hence, by Montel's
theorem $\left[ 1\mbox{, p. 165}\right] $, there exists a uniformly
convergent subsequence $\left\{ g_{n_{k}}(z):k=1,2,\right\} $ to an analytic
function, say $g(z)$, on $S_{a,b}$. The function $g(z)$, by (2.1), is not
identically $0$ on $S_{a,b}$ and, because of (2.5) and (2.4), has a zero at
the point $z_{0}$ which, by analyticity, is not isolated. Then we can
determine a positive number 
\[
r<\min \left\{ x_{0}-a\mbox{, }b-x_{0}\mbox{, }\epsilon \mbox{, }\frac{1}{2}%
\right\} 
\]%
such that $\left\vert g(z)\right\vert \neq 0$ for all $z$ on the circle $%
\left\vert z-z_{0}\right\vert =r$, so the number 
\[
m:=\min \left\{ \left\vert g(z)\right\vert :\left\vert z-z_{0}\right\vert
=r\right\} >0\mbox{.} 
\]%
Noticing $\left\vert z-z_{0}\right\vert =r$ is a compact set in the strip $%
S_{a,b}$ and $\lim_{k\rightarrow \infty }g_{n_{k}}(z)=g(z)$ is uniform on
the compacts of $S_{a,b}$, there exists a positive integer $k_{0}$ such that 
\[
\left\vert g(z)-g_{n_{k}}(z)\right\vert <m\leq \left\vert g(z)\right\vert 
\mbox{ for all }\left\vert z-z_{0}\right\vert =r\mbox{ and for all }k>k_{0}%
\mbox{.} 
\]%
This means, by Rouch\'{e}'s theorem $\left[ 1\mbox{, p. 88}\right] $, that
every function $g_{n_{k}}(z)$, for all $k>k_{0}$, has in the disk $%
\left\vert z-z_{0}\right\vert <r$ as many zeros as $g(z)$ does, so at least
one, say $z_{n_{k}}$. Then, from (2.5), one has 
\[
f(z_{n_{k}}+iT_{n_{k}})=0\mbox{ for all }k>k_{0}\mbox{.} 
\]%
Finally, since $\left\vert z_{n_{k}}-z_{0}\right\vert <r<$ $\frac{1}{2}$
for all $k>k_{0}$ and $T_{n_{k+1}}-T_{n_{k}}>1$ for all $k$, the points $%
z_{n_{k}}+iT_{n_{k}}$ are distinct for all $k>k_{0}$. Hence, the function $%
f(z)$ has a zero at each point $z_{n_{k}}+iT_{n_{k}}$ for any $k>k_{0}$.
Consequently, $f(z)$ has infinitely many zeros in the strip $S_{\epsilon }$.
This completes the proof.
\end{proof}

Since the partial sums of the Riemann zeta function are almost-periodic
functions, from the previous lemma, we deduce a relevant property on the
distribution of their zeros with respect to the imaginary axis.

\begin{theorem}
Every partial sum of the Riemann zeta function, $\zeta
_{n}(z):=\sum_{k=1}^{n}\frac{1}{k^{z}}$, $n\geq 2$, possesses infinitely
many zeros in each half-plane $\left\{ z:\Re z<0\right\} $ and $\left\{
z:\Re z>0\right\} $ except for $n=2$ whose zeros are all on the imaginary
axis.
\end{theorem}

\begin{proof}
For any $n\geq 2$, by regarding the functions 
\[
G_{n}(z):=\zeta _{n}(-z)\mbox{, }n\geq 2\mbox{, } 
\]%
the sets of zeros of $G_{n}(z)$ and $\zeta _{n}(z)$, denoted by $%
Z_{G_{n}(z)} $ and $Z_{\zeta _{n}(z)}$ respectively, are related by $%
Z_{G_{n}(z)}=-Z_{\zeta _{n}(z)}$. Then the bounds 
\[
a_{n}:=\inf \left\{ \Re z:G_{n}(z)=0\right\} \mbox{, }b_{n}:=\sup \left\{
\Re z:G_{n}(z)=0\right\} 
\]%
and 
\[
a^{(n)}:=\inf \left\{ \Re z:\zeta _{n}(z)=0\right\} \mbox{, }b^{(n)}:=\sup
\left\{ \Re z:\zeta _{n}(z)=0\right\} \mbox{,} 
\]%
satisfy 
\begin{equation}\label{2.6}
a^{(n)}=-b_{n}\mbox{, }b^{(n)}=-a_{n}\mbox{ for all }n\geq 2\mbox{.} 
%\tag{2.6}
\end{equation}%
Concerning the sets $Z_{G_{n}(z)}$ it is needed to remind that in $\left[ 13%
\mbox{, Propositions 1,2,3}\right] $ was proved:

a) For any $n\geq 2$, $Z_{G_{n}(z)}$ is infinite.

b) For every $n\geq 2$, there exist real numbers $r_{n}$, $s_{n}$, such that 
$Z_{G_{n}(z)}$ is contained in the vertical strip $\left\{ z\in 
%TCIMACRO{\U{2102} }%
%BeginExpansion
\mathbb{C}
%EndExpansion
:r_{n}\leq \Re z\leq s_{n}\right\} $.

c) For any $n\geq 2$, $Z_{G_{n}(z)}$ is not contained in the imaginary axis,
except for $n=2$.

As we have seen in (1.6) the zeros of $\zeta _{2}(z)$ are all imaginary,
then the last part of the theorem follows. Hence, from now on, we assume
that $n>2$. Let suppose that there exists an integer $m>2$ such that the
corresponding function $G_{m}(z)$ has its zeros, say $(\alpha _{m,l})_{l\geq
1}$, of the form \ 
\begin{equation}\label{2.7}
\Re \alpha _{m,l}\neq 0\mbox{ for }l=1,...,p\mbox{, }\Re \alpha _{m,l}=0%
\mbox{ if }l>p\mbox{, for some integer }p\geq 1\mbox{.}  %\tag{2.7}
\end{equation}%
Noticing $G_{n}(\overline{z})=\overline{G_{n}(z)}$ for all $z\in 
%TCIMACRO{\U{2102} }%
%BeginExpansion
\mathbb{C}
%EndExpansion
$, $(\alpha _{m,l})_{l\geq 1}$ are conjugate. Then $p$ is necessarily even,
so $p=2q\geq 2$, and the zeros $\alpha _{m,l}$ with $l>p$ are of the form $%
\pm iy_{j}$ with $y_{j}>0$. Let denote by $P(z)$ the polynomial with leader
coefficient $1$ and zeros $(\alpha _{m,l})_{l=1,2,...p}$. Then 
\begin{equation}\label{2.8}
P(z)=(z^{2}-2a_{1}z+\left\vert \alpha _{m,1}\right\vert ^{2})\cdot \cdot
\cdot (z^{2}-2a_{q}z+\left\vert \alpha _{m,q}\right\vert ^{2})\mbox{,} 
%\tag{2.8}
\end{equation}%
where $a_{l}:=\Re \alpha _{m,l}$, $l=1,...,q$. Therefore the function 
\[
H_{m}(z):=\frac{G_{m}(z)}{P(z)} 
\]%
is entire of order $1$ with zeros $(\pm iy_{j})_{j=1,2,...}$ . By Hadamard's
factorization theorem $\left[ 1\mbox{, Th. 4.4.3}\right] $ we can write 
\begin{equation}\label{2.9}
H_{m}(z)=e^{Az+B}\prod_{j=1}^{\infty }\left( 1+\frac{z^{2}}{y_{j}^{2}}%
\right) \mbox{,}  %\tag{2.9}
\end{equation}%
where the constants $B$ and $A$, because of (2.8), satisfy 
\[
e^{B}=H_{m}(0)=\frac{m}{\left\vert \alpha _{m,1}...\alpha _{m,q}\right\vert
^{2}}\mbox{,} 
\]%
\[
A=\frac{H_{m}^{\prime }(0)}{e^{B}}=\frac{\log (m!)}{m}+2\left( \frac{a_{1}}{%
\left\vert \alpha _{m,1}\right\vert ^{2}}+...+\frac{a_{q}}{\left\vert \alpha
_{m,q}\right\vert ^{2}}\right) \mbox{.} 
\]%
By (2.9) the function $H_{m}(z)e^{-Az}$ is even, then 
\[
H_{m}(z)e^{-Az}=H_{m}(-z)e^{Az}\mbox{ for all }z\in 
%TCIMACRO{\U{2102} }%
%BeginExpansion
\mathbb{C}
%EndExpansion
\mbox{.} 
\]%
Now, substituting $H_{m}(z)$ by $\frac{G_{m}(z)}{P(z)}$, we get 
\begin{equation}\label{2.10}
\frac{P(-z)}{P(z)}=\frac{G_{m}(-z)e^{2Az}}{G_{m}(z)}\mbox{ for all }z\in 
%TCIMACRO{\U{2102} }%
%BeginExpansion
\mathbb{C}
%EndExpansion
\mbox{.}  %\tag{2.10}
\end{equation}%
If we assume that any zero of $P(z)$ is a zero of $P(-z)$, the polynomials $%
P(z)$ and $P(-z)$ are equal and their zeros $(\alpha _{m,l})_{l=1,2,...p}$
are conjugate and opposite. Hence $q=\frac{p}{2}$ is even, 
\[
\frac{a_{1}}{\left\vert \alpha _{m,1}\right\vert ^{2}}+...+\frac{a_{q}}{%
\left\vert \alpha _{m,q}\right\vert ^{2}}=0 
\]%
and then $A=\frac{\log (m!)}{m}$. But this leads us to the following
contradiction: noticing $(n!)^{2}>n^{n}$ for all $n>2$,\ the right side of
(2.10) tends to $+\infty $ as $z\in 
%TCIMACRO{\U{211d} }%
%BeginExpansion
\mathbb{R}
%EndExpansion
$ does, whereas its left side is identically $1$. Consequently, at least\
one zero of $P(z)$ cannot be a zero of $P(-z)$. Thus the left side of (2.10)
is a meromorphic function having at least a pole. On the other hand, the
right side of (2.10) is a quotient of exponential polynomials with a finite
number of poles (the zeros of $P(z)$ that are not cancelled with the zeros
of $P(-z)$). Then Shields's theorem $\left[ 21\right] $ asserts that $\frac{%
G_{m}(-z)e^{2Az}}{G_{m}(z)}$ is an exponential polynomial, which represents
a new contradiction because $\frac{P(-z)}{P(z)}$ has at least a pole. As
consequence, there is no function $G_{m}(z)$, $m>2$, having its zeros of the
form (2.7). Therefore, taking into account c), every function $G_{n}(z)$, $%
n>2$, has infinitely many zeros in at least one of the half-planes $\left\{
z:\Re z<0\right\} $ and $\left\{ z:\Re z>0\right\} $. However $G_{n}(z)$, $%
n>2$, cannot have all its zeros, say $(z_{n,l})_{l=1,2,...}$, with real part
of the same sign. Indeed, by expressing $G_{n}(z)=1+2^{z}+...+n^{z}$ of the
form 
\[
1+e^{z\log 2}+...+e^{z\log n}\mbox{,} 
\]%
by Ritt's formula $\left[ 19,\mbox{ formula (9)}\right] $, we have 
\begin{equation}\label{2.11}
\sum_{l=1}^{\infty }\Re z_{n,l}=O(1)\mbox{.}  %\tag{2.11}
\end{equation}%
Now, by assuming, for instance, $\Re z_{n,l}\geq 0$ for all $l\geq 1$, from
c), there exists at least a zero of $G_{n}(z)$, say $z_{n,l_{0}}$, such that 
$\Re z_{n,l_{0}}>0$. By taking $\epsilon =\Re z_{n,l_{0}}$ and noting that $%
G_{n}(z)$ is an almost-periodic function, because of Lemma 1, $G_{n}(z)$ has
infinitely many zeros in the strip 
\[
S_{\frac{\epsilon }{2}}:=\left\{ z:\frac{\epsilon }{2}<\Re z<\frac{3\epsilon 
}{2}\right\} \mbox{.} 
\]%
This implies that%
\[
\sum_{l=1}^{\infty }\Re z_{n,l}=+\infty \mbox{,} 
\]%
contradicting (2.11). Hence, $G_{n}(z)$ has at least a zero in $\left\{
z:\Re z<0\right\} $ and again by using Lemma 1, $G_{n}(z)$ has infinitely
many zeros with negative real part. Then the theorem follows for the
functions $G_{n}(z)$ and, taking into account (2.6), it is also true for the
partial sums $\zeta _{n}(z)$.
\end{proof}

The next result is an easy consequence of the above theorem.

\begin{corollary}
The bounds $a_{n}$, $b_{n}$ and $a^{(n)}$, $b^{(n)}$ corresponding to the
functions $G_{n}(z)$ and $\zeta _{n}(z)$, respectively, satisfy 
\[
a_{n}=b_{n}=a^{(n)}=b^{(n)}=0\mbox{ if }n=2 
\]%
and 
\[
a_{n}<0\mbox{, }b_{n}>0\mbox{, }a^{(n)}<0\mbox{, }b^{(n)}>0\mbox{ for all }%
n>2\mbox{.} 
\]
\end{corollary}

\section{The level curves $\left\vert G_{n}^{\ast }(z)\right\vert
=p_{k_{n}}^{x_{0}}$, $n>2$, $x_{0}\in 
%TCIMACRO{\U{211d} }%
%BeginExpansion
\mathbb{R}
%EndExpansion
$}

The curves $\left\vert f(z)\right\vert =$constant, where $f(z)$ is an
analytic function on a given region of the complex plane, are defined as
level curves (see for instance Titchmarsh's book $\left[ 25\mbox{, }p.121%
\right] $). This concept will be crucial for attaining the main objective of
our paper, namely, to prove the existence of non-degenerate intervals
contained in the sets 
\begin{equation}\label{3.1}
R_{n}:=\overline{\left\{ \Re z:G_{n}(z)=0\right\} }  %\tag{3.1}
\end{equation}%
and%
\begin{equation}\label{3.2}
R^{(n)}:=\overline{\left\{ \Re z:\zeta _{n}(z)=0\right\} }\mbox{,}  %\tag{3.2}
\end{equation}%
for all integer $n>2$.

From the definition of the bounds corresponding to the functions $G_{n}(z)$, 
$\zeta _{n}(z)$, and (3.1), (3.2), it follows 
\[
a_{n}\mbox{, }b_{n}\in R_{n}\subset \left[ a_{n},b_{n}\right] 
\]%
and 
\[
a^{(n)}\mbox{, }b^{(n)}\in R^{(n)}\subset \left[ a^{(n)},b^{(n)}\right] 
\mbox{, } 
\]%
for all $n\geq 2$. By Corollary 3, the set $R_{n}=R^{(n)}=\left\{ 0\right\} $
if and only if $n=2$. Therefore $n=2$ is the trivial case and whether $n>2$,
the sets $R_{n}$, $R^{(n)}$ have at least two distinct points, namely, the
bounds $a_{n}$, $b_{n}$ and $a^{(n)}$, $b^{(n)}$, respectively.

For every $n>2$ we define the function 
\[
G_{n}^{\ast }(z):=G_{n}(z)-p_{k_{n}}^{z}\mbox{,} 
\]%
where $G_{n}(z):=1+2^{z}+...+n^{z}$ and $p_{k_{n}}$ is the last prime not
exceeding $n$. The goal is to study properties of the level curves 
\begin{equation}\label{3.3}
\left\vert G_{n}^{\ast }(z)\right\vert =p_{k_{n}}^{x_{0}}\mbox{,}  %\tag{3.3}
\end{equation}%
for each integer $n>2$ and a given real number $x_{0}$.

In order to emphasize the importance of the level curves, we recall that\
the sets $R_{n}$, defined in (3.1), were characterized in $\left[ 6\right] $
in terms of such curves as follows.

\begin{theorem}
(Dubon, Mora, Sepulcre, Ubeda, Vidal $\left[ 6\mbox{, Th. 2}\right] $) 
\textit{A real number }$x_{0\mbox{ }}$belongs to\textit{\ }$R_{n}$, $n>2$,%
\textit{\ if and only if the vertical line }$x=x_{0}$ \textit{intersects the
level curve }$\left\vert G_{n}^{\ast }(z)\right\vert =p_{k_{n}}^{x_{0}}$%
\textit{.}
\end{theorem}

By defining the real functions

\[
A_{n}(x,y):=\left\vert G_{n}^{\ast }(x+iy)\right\vert -p_{k_{n}}^{x}\mbox{; }%
x,y\in 
%TCIMACRO{\U{211d} }%
%BeginExpansion
\mathbb{R}
%EndExpansion
\mbox{, }n>2\mbox{,} 
\]%
the above result was completed in $\left[ 15\right] $ under the form:

\begin{theorem}
(Mora $\left[ 15\mbox{, Th. 2}\right] $) For every integer $n>2$, a real
number $x$ belongs to $R_{n}$ if and only if there exists\ some $y\in 
%TCIMACRO{\U{211d} }%
%BeginExpansion
\mathbb{R}
%EndExpansion
$ such that $A_{n}(x,y)=0$. Furthermore, $A_{n}(x,0)\geq 0$ for all $x\in $ $%
\left[ a_{n},b_{n}\right] $.
\end{theorem}

By squaring (3.3) we obtain the cartesian equation of the level curve for a
given real number\ $x_{0}$, 
\[
1+2^{2x}+...+n^{2x}+2.1^{x}\sum_{m=2}^{n}m^{x}\cos \left( y\log (\frac{m}{1}%
)\right) +2.2^{x}\sum_{m=3}^{n}m^{x}\cos \left( y\log (\frac{m}{2})\right) + 
\]%
\begin{equation}\label{3.4}
+...+2(n-1)^{x}\sum_{m=n}^{n}m^{x}\cos \left( y\log (\frac{m}{n-1})\right)
=p_{k_{n}}^{2x_{0}}\mbox{, }  %\tag{3.4}
\end{equation}%
where in the left side of (3.4) the integer variables do not take the value $%
p_{k_{n}}$.

Since, for any value of $y$, the limit of the left side of (3.4) is $+\infty 
$ when $x\rightarrow +\infty $ and the rigth side of (3.4) is fixed, the
range of $x$ is always upper bounded for every $x_{0}$. However, as the
limit of the left side of (3.4) is $1$ when $x\rightarrow -\infty $, if $%
x_{0}\neq 0$, the range of the variable $x$ must be lower bounded. Finally,
from equation (3.4), it is immediate that the level curve neither contains a
vertical segment nor has a vertical asymptote. Thus, fixed $n>2$ and $%
x_{0}\in 
%TCIMACRO{\U{211d} }%
%BeginExpansion
\mathbb{R}
%EndExpansion
$, the variable $x$ in the equation (3.4) either runs on a finite interval 
\[
\left[ b_{n,x_{0}}^{-},b_{n,x_{0}}\right] \mbox{, if }x_{0}\neq 0\mbox{,} 
\]%
or on the infinite interval 
\[
\left( -\infty ,b_{n,x_{0}}\right] \mbox{, if }x_{0}=0\mbox{,} 
\]%
where the extremes $b_{n,x_{0}}^{-}$, $b_{n,x_{0}}$ are defined as follows.

\begin{definition}
Fixed an integer $n>2$ and a real number $x_{0}$, we define $b_{n,x_{0}}$,
called upper extreme associated to $x_{0}$, as the maximum value that we can
assign to the variable $x$ to have, for at least a value of $y$, a point $%
\left( b_{n,x_{0}},y\right) $ of the level curve $\left\vert G_{n}^{\ast
}(z)\right\vert =p_{k_{n}}^{x_{0}}$. When $x_{0}\neq 0$, we define $%
b_{n,x_{0}}^{-}$, called lower extreme associated to $x_{0}$, as the minimum
value that we can assign to the variable $x$ to have, for at least a value
of $y$, a point $\left( b_{n,x_{0}}^{-},y\right) $ of the level curve $%
\left\vert G_{n}^{\ast }(z)\right\vert =p_{k_{n}}^{x_{0}}$.
\end{definition}

To illustrate the above definition we determine in the next example the
points extremes corresponding to the level curve $\left\vert G_{n}^{\ast
}(z)\right\vert =p_{k_{n}}^{x_{0}}$ for $n=3$.

\begin{example}
The level curve $\left\vert G_{3}^{\ast }(z)\right\vert =p_{k_{3}}^{x_{0}}$, 
$x_{0}\in 
%TCIMACRO{\U{211d} }%
%BeginExpansion
\mathbb{R}
%EndExpansion
$.
\end{example}

For $n=3$ one has $p_{k_{3}}=3$, then $G_{3}^{\ast
}(z):=G_{3}(z)-p_{k_{3}}^{z}=1+2^{z}$. Hence the level curve is $\left\vert
1+2^{z}\right\vert =3^{x_{0}}$. By squaring we obtain its cartesian
equation, 
\begin{equation}\label{3.5}
1+2^{2x}+2^{x+1}\cos (y\log 2)=3^{2x_{0}}\mbox{.}  %\tag{3.5}
\end{equation}%
The lower extreme associated to $x_{0}\neq 0$ is 
\[
b_{3,x_{0}}^{-}=\left\{ 
\begin{array}{c}
\frac{\log (1-3^{x_{0}})}{\log 2}\mbox{ if }x_{0}<0 \\ 
\frac{\log (3^{x_{0}}-1)}{\log 2}\mbox{ if }x_{0}>0%
\end{array}%
\right. 
\]%
and the upper extreme associated to any $x_{0}$ is 
\[
b_{3,x_{0}}=\frac{\log (1+3^{x_{0}})}{\log 2}\mbox{.} 
\]%
If $x_{0}=0$, by dividing (3.5) by $2^{x+1}$, the equation of the level
curve is 
\[
2^{x-1}+\cos (y\log 2)=0 
\]%
and then $x$ runs on $\left( -\infty ,1\right] $.

Thus, for $x_{0}<0$, (3.5) represents an infinite family of closed curves
containing each of them one zero of $G_{3}^{\ast }(z)$, so the level curve
has infinitely many arc-connected components. For $x_{0}=0$, the above
equation represents an infinite family of open curves having horizontal
asymptotes oriented to $-\infty $ of equations $y=\left( 2k+1\right) \frac{%
\pi }{2\log 2}$, $k\in 
%TCIMACRO{\U{2124} }%
%BeginExpansion
\mathbb{Z}
%EndExpansion
$, so the level curve has infinitely many arc-connected components. Finally,
for $x_{0}>0$, (3.5) is the equation of a unique open curve, that is the
level curve has only one arc-connected component. The curve meets to the
real axis at the point $b_{3,x_{0}}^{-}=\frac{\log (3^{x_{0}}-1)}{\log 2}$,
and the variable $y$ takes all real numbers.

The next result proves that the level curve $\left\vert G_{n}^{\ast
}(z)\right\vert =p_{k_{n}}^{x_{0}}$ for $n>3$ is essentially the same as the
case $n=3$.

\begin{proposition}
Fixed an integer $n>2$ and a real number $x_{0}$, let us consider the level
curve $\left\vert G_{n}^{\ast }(z)\right\vert =p_{k_{n}}^{x_{0}}$. Then, if $%
x_{0}<0$, the level curve has infinitely many arc-connected components which
are closed curves; if $x_{0}=0$, the level curve has infinitely many
arc-connected components which are open curves with horizontal asymptotes,
as $x\rightarrow -\infty $, of equations $y=\left( 2k+1\right) \frac{\pi }{%
2\log 2}$, $k\in 
%TCIMACRO{\U{2124} }%
%BeginExpansion
\mathbb{Z}
%EndExpansion
$. Moreover, the level curve only meets to the real axis at a point, say $%
c_{n,x_{0}}$, if and only if $x_{0}>0$; in this case the lower extreme $%
b_{n,x_{0}}^{-}=$ $c_{n,x_{0}}$ and the level curve has only one
arc-connected component which is an open curve where the variable $y$ takes
all real values.
\end{proposition}

\begin{proof}
Because of the function of the left side of (3.4) is even with respect to
the variable $y$, we will assume that $y\geq 0$. For a given a real number $%
x_{0}\leq 0$ we write (3.4) as 
\[
2^{2x}+...+n^{2x}+2.1^{x}\sum_{m=2}^{n}m^{x}\cos \left( y\log (\frac{m}{1}%
)\right) +2.2^{x}\sum_{m=3}^{n}m^{x}\cos \left( y\log (\frac{m}{2})\right) + 
\]%
\begin{equation}\label{3.6}
+...+2(n-1)^{x}\sum_{m=n}^{n}m^{x}\cos \left( y\log (\frac{m}{n-1})\right)
=p_{k_{n}}^{2x_{0}}-1  %\tag{3.6}
\end{equation}%
and, since $p_{k_{n}}^{2x_{0}}-1\leq 0$, the variable $y$ satisfies either 
\begin{equation}\label{3.7}
\frac{\pi }{2}<y\log n\mbox{, if }n\mbox{ is a composite number,}  %\tag{3.7}
\end{equation}%
or 
\begin{equation}\label{3.8}
\frac{\pi }{2}<y\log (n-1)\mbox{, if }n\mbox{ is prime.}  %\tag{3.8}
\end{equation}%
(Otherwise the left side of (3.6) would be positive). Hence, by (3.7), or
(3.8), the level curve $\left\vert G_{n}^{\ast }(z)\right\vert
=p_{k_{n}}^{x_{0}}$, for a given a real number $x_{0}\leq 0$, has no point
in the strip $\left\{ (x,y):0\leq y\leq \frac{\pi }{2\log n}\right\} $,
when $n$ is composite, and it has no point in the strip $\left\{ (x,y):0\leq
y\leq \frac{\pi }{2\log (n-1)}\right\} $, if $n$ is prime.

Now assume that $x_{0}<0$. By writing $\left\vert G_{n}^{\ast
}(z)\right\vert =p_{k_{n}}^{x_{0}}$ as $\left\vert 1+w\right\vert =r_{x_{0}}$%
, where $w:=2^{z}+...+n^{z}$ (this expression does not contain the term $%
p_{k_{n}}^{z}$) and $r_{x_{0}}:=p_{k_{n}}^{x_{0}}$, the equation (3.6)
describes the continuous motion of each zero $z=(x,y)$ of the function 
\[
F_{n,w}^{\ast }(z):=G_{n}^{\ast }(z)-1-w 
\]%
when $w$ runs on the circle $\left\vert 1+w\right\vert =r_{x_{0}}$. Indeed,
fixed $n>2$, for each $w$ of the circle $\left\vert 1+w\right\vert
=r_{x_{0}} $ the function $F_{n,w}^{\ast }(z)$ has infinitely many zeros on
the half-plane\ $\left\{ z:\Im z\geq 0\right\} $ situated on a vertical
strip. These zeros will be supposed lexicografically ordered. That is, given
a point $w$ of the circle $\left\vert 1+w\right\vert =r_{x_{0}}$, a zero $%
z_{w}=(x,y)$ of $F_{n,w}^{\ast }(z)$ precedes to other zero $z_{w}^{\prime
}=(x^{\prime },y^{\prime })$ if and only if either $y<y^{\prime }$ or $%
x<x^{\prime }$, if $y=y^{\prime }$. Then, by taking the interval $\left(
0,2\pi \right] $ as a determination of the argument (the domain
corresponding to the $\log $ function is $\Omega :=%
%TCIMACRO{\U{2102} }%
%BeginExpansion
\mathbb{C}
%EndExpansion
\setminus \left[ 0,+\infty \right) $), let us define the map%
\[
\varphi _{1}(w):=z_{1,w}\mbox{,} 
\]%
that assigns to each $w$ of the circle $\left\vert 1+w\right\vert =r_{x_{0}}$
the first zero, $z_{1,w}$, of $F_{n,w}^{\ast }(z)$ with respect to the order
defined. Because of $0<r_{x_{0}}<1$, the circle $\left\vert 1+w\right\vert
=r_{x_{0}}$ is in $\Omega $ and the argument of $w$ varies on a certain
interval $\left[ \pi -\alpha _{x_{0}},\pi +\alpha _{x_{0}}\right] $ for some 
$\alpha _{x_{0}}$ such that\ $0<\alpha _{x_{0}}<\frac{\pi }{4}$, so $\left[
\pi -\alpha _{x_{0}},\pi +\alpha _{x_{0}}\right] \subset \left( 0,2\pi %
\right] $. Thus, the function $\varphi _{1}$ transforms the circle $%
\left\vert 1+w\right\vert =r_{x_{0}}$ on a closed curve whose equation is
given by (3.6) and it constitutes the first arc-connected component of the
level curve. By analogy, the function 
\[
\varphi _{j}(w):=z_{j,w}\mbox{, }j>1\mbox{,} 
\]%
where $z_{j,w}$ is the $jth$ zero of $F_{n,w}^{\ast }(z)$ with respect to
the order defined, transforms the circle $\left\vert 1+w\right\vert
=r_{x_{0}}$ on a closed curve, whose equation is given by (3.6), and it is
the $jth$ arc-connected component of the level curve, and so on. Finally, as
the function of the left side of equation (3.6) is even with respect to the
variable $y$, we obtain by simmetry with respect to the real axis the closed
arc-connected components of the level curve in the half-plane $\left\{ z:\Im
z<0\right\} $.

If $x_{0}=0$, the right side of (3.6) is $0$. Then dividing (3.6) by $%
2^{x+1} $, the equation of the level curve is 
\[
2^{x-1}+...+\left( \frac{n}{2}\right) ^{2x}2^{x-1}+\cos (y\log
2)+1^{x}\sum_{m=3}^{n}\left( \frac{m}{2}\right) ^{x}\cos \left( y\log (\frac{%
m}{1})\right) + 
\]%
\[
+2^{x}\sum_{m=3}^{n}\left( \frac{m}{2}\right) ^{x}\cos \left( y\log (\frac{m%
}{2})\right) +...+ 
\]%
\begin{equation}\label{3.9}
+(n-1)^{x}\sum_{m=n}^{n}\left( \frac{m}{2}\right) ^{x}\cos \left( y\log (%
\frac{m}{n-1})\right) =0\mbox{.}  %\tag{3.9}
\end{equation}%
By taking the limit as $x\rightarrow -\infty $, all the terms of the left
side of (3.9), except $\cos \left( y\log 2\right) $, tend to $0$. Then the
equation $\cos \left( y\log 2\right) =0$ gives us the horizontal asymptotes
of the level curve. Consequently, the level curve has infinitely many
arc-connected components which are open curves (closed on the right) having
asymptotes of equations $y=\left( 2k+1\right) \frac{\pi }{2\log 2}$, $k\in 
%TCIMACRO{\U{2124} }%
%BeginExpansion
\mathbb{Z}
%EndExpansion
$, oriented to $-\infty $.

For studying the case $x_{0}>0$, let $y_{0}$ be a fixed real number. By
substituting it in (3.4), we have an expression of the form 
\begin{equation}\label{3.10}
1+a_{2}2^{x}+a_{3}3^{x}+...+a_{n(n-1)}\left( n(n-1)\right)
^{x}+a_{n^{2}}n^{2x}=p_{k_{n}}^{2x_{0}}\mbox{,}  %\tag{3.10}
\end{equation}%
where the $a$'s are real numbers (the left side of (3.10) does not contain
the term $a_{p_{k_{n}}}p_{k_{n}}^{x}$). Because the left side of (3.10) has
been obtained from (3.3) as the sum of the squares of the real and imaginary
parts of $G_{n}^{\ast }(z)$, for $z=x+iy_{0}$, the number of changes of the
sign of its coefficients is a nonnegative even integer. Hence, since $%
p_{k_{n}}^{2x_{0}}>1$, the number, say $W$, of changes of the sign of the
coefficients of the equation, equivalent to (3.10), 
\[
1-p_{k_{n}}^{2x_{0}}+a_{2}2^{x}+a_{3}3^{x}+...+a_{n(n-1)}\left(
n(n-1)\right) ^{x}+a_{n^{2}}n^{2x}=0\mbox{,} 
\]%
is odd. Then, by P\'{o}lya and Sz\"{e}go's formula $\left[ 18\mbox{, p. 46}%
\right] $, $W-N$ is an even nonnegative integer, where $N$ denotes the
number of real zeros (counting multiplicities) of the preceding equation.
Therefore, $N$ must be positive, which implies that equation (3.10) has
solution. Consequently, as $y_{0}$ is arbitrary, the variable $y$ takes all
the real values. On the other hand, by making $y=0$ in (3.4), one has 
\[
1+2^{x}+...+\mbox{ }n^{x}=p_{k_{n}}^{x_{0}} 
\]%
(the left side of this equation does not contain the term $p_{k_{n}}^{x}$)
and we then claim that the above equation has real solution if and only if $%
x_{0}>0$. Indeed, clearly the equation has no real solution if $x_{0}\leq 0$%
. Then if the equation has solution, $x_{0}$ must be positive. Conversely,
by assuming $x_{0}>0$, the number $p_{k_{n}}^{x_{0}}$ is greater than $1$
and\ $W$, the number of changes of the sign of the equation 
\[
1-p_{k_{n}}^{x_{0}}+2^{x}+...+\mbox{ }n^{x}=0\mbox{,} 
\]%
is $1$. Thus, by P\'{o}lya and Sz\"{e}go's formula $\left[ 18\right] $, the
number, $N$, of real solutions of the above equation is necessarily $1$.
Hence the above equation has only a real solution, say $c_{n,x_{0}}$, and
the claim follows. Furthermore, since $(c_{n,x_{0}},0)$ is a point of the
level curve, from Definition 6, one has 
\[
b_{n,x_{0}}^{-}\leq c_{n,x_{0}}\mbox{.} 
\]%
But, if we suppose that $b_{n,x_{0}}^{-}<c_{n,x_{0}}$, again from Definition
6, for some $y$, the point $\left( b_{n,x_{0}}^{-},y\right) $ belongs to the
level curve and then $\left\vert G_{n}^{\ast
}(b_{n,x_{0}}^{-}+iy)\right\vert =p_{k_{n}}^{x_{0}}$. Now, because $%
G_{n}^{\ast }(z)$ is strictely increasing on the real axis, we obtain 
\[
p_{k_{n}}^{x_{0}}=\left\vert G_{n}^{\ast }(b_{n,x_{0}}^{-}+iy)\right\vert
\leq \left\vert G_{n}^{\ast }(b_{n,x_{0}}^{-})\right\vert <\left\vert
G_{n}^{\ast }(c_{n,x_{0}})\right\vert =p_{k_{n}}^{x_{0}}\mbox{,} 
\]%
a contradiction. Therefore, 
\[
b_{n,x_{0}}^{-}=c_{n,x_{0}}\mbox{.} 
\]%
Finally, since the variable\ $y$ takes all real numbers, the unique point
that the level curve meets to the real axis is $(b_{n,x_{0}}^{-},0)$ and
there is no vertical asymptote, the level curve $\left\vert G_{n}^{\ast
}(z)\right\vert =p_{k_{n}}^{x_{0}}$ has only one arc-connected component.
The proof is now completed.
\end{proof}

From Example 7 and Proposition 8 we obtain the following consequence.

\begin{corollary}
For a given integer $n>2$ and a real number\ $x_{0}$, let $\left\vert
G_{n}^{\ast }(z)\right\vert =p_{k_{n}}^{x_{0}}$ be the corresponding level
curve. If a point $z^{\prime }=(x^{\prime },y^{\prime })$ satisfies $%
\left\vert G_{n}^{\ast }(z^{\prime })\right\vert <p_{k_{n}}^{x_{0}}$, then
there exists a point $z=(x,y)$ of the level curve such that $x>x^{\prime }$.
\end{corollary}

\begin{proof}
By Example 7 and Proposition 8, if $x_{0}<0$, the level curve $\left\vert
G_{n}^{\ast }(z)\right\vert =p_{k_{n}}^{x_{0}}$ has infinitely many
arc-connected components which are closed curves. Thus, the level curve
divides to the complex plane in two regions $\left\{ w:\left\vert
G_{n}^{\ast }(w)\right\vert <p_{k_{n}}^{x_{0}}\right\} $ and $\left\{
w:\left\vert G_{n}^{\ast }(w)\right\vert >p_{k_{n}}^{x_{0}}\right\} $,
defined by the interior and the exterior points of its arc-connected
components, respectively. Thus, if $z^{\prime }=(x^{\prime },y^{\prime })$
is a point satisfying $\left\vert G_{n}^{\ast }(z^{\prime })\right\vert
<p_{k_{n}}^{x_{0}}$, $z^{\prime }$ is in the interior of a closed
arc-connected component and consequently it must have a point $z=(x,y)$ of
the level curve situated on the right of $z^{\prime }$, so $x>x^{\prime }$.
The case $x_{0}=0$ is analogous to the previous one because the
arc-connected components (open curves) are closed on the right side and open
on the left side with asymptotes oriented to $-\infty $. Finally, if $x_{0}>0
$, the variable $y$ take all real values and $x$ runs on the interval $\left[
b_{n,x_{0}}^{-},b_{n,x_{0}}\right] $. Then, since for any $n>2$, 
\[
\max \left\{ \left\vert G_{n}^{\ast }(w)\right\vert :\Re w\leq
b_{n,x_{0}}^{-}\right\} =G_{n}^{\ast }(b_{n,x_{0}}^{-})\mbox{,}
\]%
the set $\left\{ w:\left\vert G_{n}^{\ast }(w)\right\vert
<p_{k_{n}}^{x_{0}}\right\} $ coincides exactly with the region situated on
the left of the level curve. Hence, if $z^{\prime }=(x^{\prime },y^{\prime })
$ satisfies $\left\vert G_{n}^{\ast }(z^{\prime })\right\vert
<p_{k_{n}}^{x_{0}}$, there exists a point $z=(x,y)$ of the level curve on
the right of $z^{\prime }$ and, consequently, $x>x^{\prime }$. The proof is
now completed.
\end{proof}

The next result shows that the set $R_{n}:=\overline{\left\{ \Re %
z:G_{n}(z)=0\right\} }$ contains intervals.

\begin{theorem}
Let $\left[ a_{n,}b_{n}\right] $ be the critical interval of $G_{n}(z)$, $%
n>2 $, $x_{0}\in $ $R_{n}$ and $b_{n,x_{0}}$ its upper extreme associated
with respect to the level curve $\left\vert G_{n}^{\ast }(z)\right\vert
=p_{k_{n}}^{x_{0}}$. Then 
\[
\left[ x_{0},b_{n,x_{0}}\right] \cap \left[ a_{n,}b_{n}\right] \subset R_{n}%
\mbox{.} 
\]
\end{theorem}

\begin{proof}
Since $x_{0}\in $ $R_{n}$, because of Theorem 4, there exists a value $y_{0}$
such that the point $\left( x_{0},y_{0}\right) $ lies on the level curve $%
\left\vert G_{n}^{\ast }(z)\right\vert =p_{k_{n}}^{x_{0}}$. From Definition
6, if $x_{0}\neq 0$, then $x_{0}\in \left[ b_{n,x_{0}}^{-},b_{n,x_{0}}\right]
$ and $x_{0}\in \left( -\infty ,b_{n,x_{0}}\right] $, if $x_{0}=0$.
Therefore, for any $x_{0}\in $ $R_{n}$, one has $x_{0}\leq b_{n,x_{0}}$.
Thus, if $x_{0}=b_{n,x_{0}}$, since $R_{n}\subset \left[ a_{n},b_{n}\right] $%
, the theorem trivially follows. Let us assume that $x_{0}<b_{n,x_{0}}$ and
let $x_{1}\in \left[ a_{n},b_{n}\right] $ be a point such that $%
x_{0}<x_{1}<b_{n,x_{0}}$. Then $x_{1}\in \left[ b_{n,x_{0}}^{-},b_{n,x_{0}}%
\right] $, or $x_{1}\in \left( -\infty ,b_{n,x_{0}}\right] $, so the line of
equation $x=x_{1}$ intersects to the level curve $\left\vert G_{n}^{\ast
}(z)\right\vert =p_{k_{n}}^{x_{0}}$ at least at a point, say $%
z_{1}=x_{1}+iy_{1}$. Consequently $\left\vert G_{n}^{\ast
}(z_{1})\right\vert =p_{k_{n}}^{x_{0}}$. Now, by the definition of $%
A_{n}(x,y)$, we obtain 
\[
A_{n}(x_{1},y_{1})=\left\vert G_{n}^{\ast }(z_{1})\right\vert
-p_{k_{n}}^{x_{1}}=p_{k_{n}}^{x_{0}}-p_{k_{n}}^{x_{1}}<0\mbox{.} 
\]%
On the other hand, since $x_{1}\in \left[ a_{n},b_{n}\right] $, by Theorem
5, one has $A_{n}(x_{1},0)\geq 0$. Then, the continuity of $A_{n}(x_{1},y)$
as function of $y$ implies the existence of a value, say $y_{2}$, such that $%
A_{n}(x_{1},y_{2})=0$. Now, because of Theorem 5, it follows that $x_{1}\in
R_{n}$ and then 
\begin{equation}\label{3.11}
\left[ x_{0},b_{n,x_{0}}\right) \cap \left[ a_{n,}b_{n}\right] \subset R_{n}%
\mbox{.}  %\tag{3.11}
\end{equation}%
By supposing $b_{n,x_{0}}>b_{n}$, from (3.11), we get 
\[
\left[ x_{0},b_{n,x_{0}}\right] \cap \left[ a_{n,}b_{n}\right] =\left[
x_{0},b_{n,x_{0}}\right) \cap \left[ a_{n,}b_{n}\right] \subset R_{n} 
\]%
and then the theorem follows. If $b_{n,x_{0}}\leq b_{n}$, since $x_{0}\in
R_{n}\subset $ $\left[ a_{n,}b_{n}\right] $, again by (3.11), one has 
\[
\left[ x_{0},b_{n,x_{0}}\right) \cap \left[ a_{n,}b_{n}\right] =\left[
x_{0},b_{n,x_{0}}\right) \subset R_{n}\mbox{,} 
\]%
but, as $R_{n}$ is closed, the point $b_{n,x_{0}}$ also belongs to $R_{n}$.
Then, $\left[ x_{0},b_{n,x_{0}}\right] \subset R_{n}$ and hence, in this
case, we also get 
\[
\left[ x_{0},b_{n,x_{0}}\right] \cap \left[ a_{n,}b_{n}\right] =\left[
x_{0},b_{n,x_{0}}\right] \subset R_{n}\mbox{,} 
\]%
which proves completely the theorem.
\end{proof}

An important conclusion on the sets $R_{3}$ and $R^{(3)}$ is deduced from
Theorem 10.

\begin{corollary}
The sets $R_{3}$ and $R^{(3)}$ are equal to $\left[ a_{3,}b_{3}\right] $ and 
$\left[ a^{(3)},b^{(3)}\right] $, respectively.
\end{corollary}

\begin{proof}
The bound $a_{3}:=\inf \left\{ \Re z:G_{3}(z)=0\right\} $ belongs to $%
R_{3}$ and, by Corollary 3, $a_{3}$ is negative. Its upper extreme
associated, with respect to the level curve $\left\vert G_{3}^{\ast
}(z)\right\vert =p_{k_{3}}^{x_{0}}$, determined in Example 7, is 
\[
b_{3,a_{3}}=\frac{\log (1+3^{a_{3}})}{\log 2}>0\mbox{.} 
\]%
Then, by taking $x_{0}=a_{3}$ in Theorem 10, we have%
\begin{equation}\label{3.12}
\left[ a_{3},\frac{\log (1+3^{a_{3}})}{\log 2}\right] \cap \left[
a_{3,}b_{3}\right] \subset R_{3}\mbox{.}  %\tag{3.12}
\end{equation}%
Hence, the corollary follows, if $\frac{\log (1+3^{a_{3}})}{\log 2}\geq
b_{3}$, because $R_{3}\subset $ $\left[ a_{3,}b_{3}\right] $. Let us assume $%
\frac{\log (1+3^{a_{3}})}{\log 2}<b_{3}$. Then, since $a_{3}<0$ and $\frac{%
\log (1+3^{a_{3}})}{\log 2}>0$, from (3.12), it deduces in particular that $%
0\in R_{3}$. But, as we have\ just seen in Example 7, the upper extreme
associated to $0$ is $1$, so, again from Theorem 10, we get 
\begin{equation}\label{3.13}
\left[ 0,1\right] \cap \left[ a_{3,}b_{3}\right] \subset R_{3}\mbox{.} 
%\tag{3.13}
\end{equation}%
From Corollary 3 and $\left[ 6\mbox{, Lemma 1}\right] $, we know that 
\[
0<b_{3}\leq \sup \left\{ x:1+2^{x}\geq 3^{x}\right\} =1\mbox{.} 
\]%
Then, by (3.13), one has 
\[
\left[ 0,1\right] \cap \left[ a_{3,}b_{3}\right] =\left[ 0,b_{3}\right]
\subset R_{3}\mbox{,} 
\]%
which, jointly with (3.12), involves that 
\[
R_{3}=\left[ a_{3},b_{3}\right] \mbox{.} 
\]%
Now, taking into account (2.6), one has $R^{(3)}=\left[ -b_{3},-a_{3}\right]
=\left[ a^{(3)},b^{(3)}\right] $. Then the proof is completed.
\end{proof}

From the above corollary it follows that the functions $G_{3}(z)$ and $\zeta
_{3}(z)$ have infinitely many zeros arbitrarily near any vertical line
contained in $\left[ a_{3},b_{3}\right] \times 
%TCIMACRO{\U{211d} }%
%BeginExpansion
\mathbb{R}
%EndExpansion
$ and $\left[ a^{(3)},b^{(3)}\right] \times 
%TCIMACRO{\U{211d} }%
%BeginExpansion
\mathbb{R}
%EndExpansion
$, respectively. Therefore, the distribution of the real parts of the zeros
of $G_{3}(z)$ and $\zeta _{3}(z)$ is completely uniform on their critical
intervals.

\begin{lemma}
For every $n>2$, let us denote by $f_{n}$ the function that assigns to each
real number $x_{0}$ its upper extreme associated $b_{n,x_{0}}$ with respect
to the level curve $\left\vert G_{n}^{\ast }(z)\right\vert
=p_{k_{n}}^{x_{0}} $. Then $f_{n}$ is a strictely increasing function.
\end{lemma}

\begin{proof}
Let us consider two real numbers $x_{1}<$ $x_{0}$. Then, since $%
p_{k_{n}}\geq 2$ for all $n\geq 2$, one has $%
p_{k_{n}}^{x_{1}}<p_{k_{n}}^{x_{0}}$. From Definition 6, given $%
f_{n}(x_{1}):=b_{n,x_{1}}$, there exists some $y_{1}$ such that the point $%
z_{1}=(b_{n,x_{1}},y_{1})$ belongs to the level curve $\left\vert
G_{n}^{\ast }(z)\right\vert =p_{k_{n}}^{x_{1}}$. Then, since $%
p_{k_{n}}^{x_{1}}<p_{k_{n}}^{x_{0}}$, the point $z_{1}$ satisfies $%
\left\vert G_{n}^{\ast }(z_{1})\right\vert <p_{k_{n}}^{x_{0}}$ . By applying
Corollary 9, there exists a point $z_{2}=(x_{2},y_{2})$ of the level curve $%
\left\vert G_{n}^{\ast }(z)\right\vert =p_{k_{n}}^{x_{0}}$ such that $%
b_{n,x_{1}}<x_{2}$. Now, as $z_{2}$ belongs to the level curve $\left\vert
G_{n}^{\ast }(z)\right\vert =p_{k_{n}}^{x_{0}}$, again from Definition 6, we
have $x_{2}\leq b_{n,x_{0}}=f_{n}(x_{0})$. Thus 
\[
f_{n}(x_{1})=b_{n,x_{1}}<x_{2}\leq b_{n,x_{0}}=f_{n}(x_{0})\mbox{,} 
\]%
so $f_{n}$ is strictely increasing.
\end{proof}

In the next result we find, for each integer $n>2$, a lower bound of the\
set of upper extremes $b_{n,x_{0}}$ when $x_{0}\in 
%TCIMACRO{\U{211d} }%
%BeginExpansion
\mathbb{R}
%EndExpansion
$.

\begin{lemma}
For every $n>2$, let us denote $b_{n}^{\ast }=\sup \left\{ \Re %
z:G_{n}^{\ast }(z)=0\right\} $. Then 
\[
b_{n}^{\ast }<b_{n,x_{0}}\mbox{ for all }x_{0}\in 
%TCIMACRO{\U{211d} }%
%BeginExpansion
\mathbb{R}
%EndExpansion
\mbox{.} 
\]
\end{lemma}

\begin{proof}
Let us assume that $G_{n}^{\ast }(z)$ has a zero, say $w$, on the vertical
line $x=b_{n}^{\ast }$. Then $\Re w=b_{n}^{\ast }$. Now, let $x_{0}$ be a
real number. Then, since $G_{n}^{\ast }(w)=0$, one has $\left\vert
G_{n}^{\ast }(w)\right\vert <p_{k_{n}}^{x_{0}}$. From Corollary 9, there
exists a point $z_{1}=\left( x_{1},y_{1}\right) $ of the level curve $%
\left\vert G_{n}^{\ast }(z)\right\vert =p_{k_{n}}^{x_{0}}$ such that $%
b_{n}^{\ast }<x_{1}$. By Definition 6 it follows that $x_{1}\leq $ $%
b_{n,x_{0}}$ and then 
\[
b_{n}^{\ast }<x_{1}\leq b_{n,x_{0}}\mbox{,} 
\]%
which proves the lemma under the above assumption.

If $G_{n}^{\ast }(z)$ has no zero on the line $x=b_{n}^{\ast }$, according
to the definition of $b_{n}^{\ast }$, there exists a sequence of zeros $%
\left( z_{m}\right) _{m}$ of $G_{n}^{\ast }(z)$ with $\Re z_{m}<$ $%
b_{n}^{\ast }$ such that $\lim_{m}\Re z_{m}=b_{n}^{\ast }$. Let us consider
a real number $x_{1}<x_{0}$. Then, as $G_{n}^{\ast }(z_{m})=0$, we get $%
\left\vert G_{n}^{\ast }(z_{m})\right\vert <p_{k_{n}}^{x_{1}}$ for all $m$.
This implies, from Corollary 9, the existence of a sequence of points of the
level curve $\left\vert G_{n}^{\ast }(z)\right\vert =p_{k_{n}}^{x_{1}}$, say 
$\left( w_{m}\right) _{m}$, such that $\Re z_{m}<\Re w_{m}$ for all $m$. On
the other hand, since every $w_{m}$ is on the level curve $\left\vert
G_{n}^{\ast }(z)\right\vert =p_{k_{n}}^{x_{1}}$, from Definition 6, one has $%
\Re w_{m}\leq b_{n,x_{1}}$ for all $m$. Hence, we have $\Re z_{m}<\Re
w_{m}\leq b_{n,x_{1}}$ for all $m$ and, taking the limit when $m\rightarrow
\infty $, we obtain 
\begin{equation}\label{3.14}
b_{n}^{\ast }\leq b_{n,x_{1}}\mbox{.}  %\tag{3.14}
\end{equation}%
Now, by Lemma 12, since $x_{1}<x_{0}$, it follows that $%
b_{n,x_{1}}<b_{n,x_{0}}$ and, noting inequality (3.14), we\ get 
\[
b_{n}^{\ast }\leq b_{n,x_{1}}<b_{n,x_{0}}\mbox{,} 
\]%
which completes the proof.
\end{proof}

The above lemma, jointly with Theorem 10, allows us to determine the sets $%
R_{n}$ and $R^{(n)}$ in the case $n=4$.

\begin{theorem}
The sets $R_{4}$ and $R^{(4)}$ are equal to $\left[ a_{4,}b_{4}\right] $ and 
$\left[ a^{(4)},b^{(4)}\right] $, respectively.
\end{theorem}

\begin{proof}
The last prime not exceeding $4$ is $3$, so $p_{k_{4}}=$ $3$. Then $%
G_{4}^{\ast }(z)=1+2^{z}+4^{z}$, and its zeros are 
\[
\frac{2\pi (3k+1)i}{3\log 2}\mbox{, }\frac{2\pi (3k+2)i}{3\log 2}\mbox{, }%
k\in 
%TCIMACRO{\U{2124} }%
%BeginExpansion
\mathbb{Z}
%EndExpansion
\mbox{.}
\]%
Hence, $b_{4}^{\ast }:=\sup \left\{ \Re z:G_{4}^{\ast }(z)=0\right\} =0$%
. By applying Theorem 10 to $x_{0}=a_{4}$, which is negative by Corollary 3,
one has 
\begin{equation}\label{3.15}
\left[ a_{4},0\right] \subset R_{4}\mbox{,}  %\tag{3.15}
\end{equation}%
because $b_{4}>0$ by virtue of Corollary 3 and $b_{4,x_{0}}>b_{4}^{\ast }=0$
by Lemma 13. From (3.4), the equation of the level curve $\left\vert
G_{4}^{\ast }(z)\right\vert =3^{x_{0}}$ is%
\[
1+2^{2x}+4^{2x}+2.2^{x}\cos (y\log 2)+2.4^{x}\cos \left( y\log 4)\right) +
\]%
\begin{equation}\label{3.16}
+2.2^{x}.4^{x}\cos \left( y\log 2)\right) =3^{2x_{0}}\mbox{.}  %\tag{3.16}
\end{equation}%
Since $\cos \left( y\log 4)\right) =2\cos ^{2}(y\log 2)-1$, from (3.16), we
obtain 
\begin{equation}\label{3.17}
\cos (y\log 2)=\frac{-(1+4^{x})\pm \sqrt{(2.3^{x_{0}})^{2}-(\sqrt{3}%
(4^{x}-1))^{2}}}{4.2^{x}}  %\tag{3.17}
\end{equation}%
and then the variable $x$ must satisfy $(\sqrt{3}(4^{x}-1))^{2}\leq
(2.3^{x_{0}})^{2}$, which is equivalent to say that 
\begin{equation}\label{3.18}
4^{x}\in \left[ 1-2.3^{x_{0}-\frac{1}{2}},1+2.3^{x_{0}-\frac{1}{2}}\right] 
\mbox{.}  %\tag{3.18}
\end{equation}%
Now, for $x_{0}\geq 0$, is $1-2.3^{x_{0}-\frac{1}{2}}<0$. Thus, noting that $%
4^{x}>0$ for any $x$, from (3.18) it follows that the variable $x$ satisfies 
\[
x\leq \frac{\log \left( 1+2.3^{x_{0}-\frac{1}{2}}\right) }{\log 4}\mbox{.}
\]%
Then, by taking $x=x_{1}:=\frac{\log \left( 1+2.3^{x_{0}-\frac{1}{2}%
}\right) }{\log 4}$ in (3.17) (observe that $(2.3^{x_{0}})^{2}-(\sqrt{3}%
(4^{x_{1}}-1))^{2}=0$), there exists a value for $y,$ given by 
\begin{equation}\label{3.19}
\cos (y\log 2)=-\frac{1+3^{x_{0}-\frac{1}{2}}}{2\left( 1+2.3^{x_{0}-\frac{1%
}{2}}\right) ^{\frac{1}{2}}}\mbox{,}  %\tag{3.19}
\end{equation}%
provided that $\left\vert -\frac{1+3^{x_{0}-\frac{1}{2}}}{2\left(
1+2.3^{x_{0}-\frac{1}{2}}\right) ^{\frac{1}{2}}}\right\vert \leq 1$. But $0<%
\frac{1+3^{x_{0}-\frac{1}{2}}}{2\left( 1+2.3^{x_{0}-\frac{1}{2}}\right) ^{%
\frac{1}{2}}}\leq 1$ if and only if $x_{0}\leq \frac{\log \left( 6+3\sqrt{3}%
\right) }{\log 3}$. Then, we claim that any $x_{0}\in \left[ 0,b_{4}\right) $%
, where $b_{4}:=\sup \left\{ \Re z:G_{4}(z)=0\right\} $, satisfies the above
condition. Indeed, because of $\left[ 6\mbox{, Lemma 1}\right] $, one has $%
b_{4}\leq \sup \left\{ x:1+2^{x}+3^{x}\geq 4^{x}\right\} $ and, on the other
hand, it is immediate to check that 
\[
\sup \left\{ x:1+2^{x}+3^{x}\geq 4^{x}\right\} <2<\frac{\log \left( 6+3%
\sqrt{3}\right) }{\log 3}\mbox{,}
\]%
so $x_{0}<b_{4}<2<\frac{\log \left( 6+3\sqrt{3}\right) }{\log 3}$, and the
claim follows. Consequently, for each $x_{0}\in \left[ 0,b_{4}\right) $, by
making $x=x_{1}$ in (3.17), there exists a value of $y$, say $y_{1}$, given
by (3.19), determining a point $\left( x_{1},y_{1}\right) $ of the level
curve $\left\vert G_{4}^{\ast }(z)\right\vert =3^{x_{0}}$. Thus, by
Definition 6, $b_{4,x_{0}}^{-}\leq x_{1}\leq b_{4,x_{0}}$. On the other
hand, from the definition of $x_{1}$, one has 
\begin{equation}\label{3.20}
x_{1}>x_{0}\mbox{ if and only if }4^{x_{0}}<1+2.3^{x_{0}-\frac{1}{2}}\mbox{.}
%\tag{3.20}
\end{equation}%
We define $\beta _{4}:=\sup \left\{ x:4^{x}\leq 1+2.3^{x-\frac{1}{2}%
}\right\} $. Then, from P\'{o}lya and Sz\"{e}go's formula $\left[ 18\mbox{,
p. 46}\right] $, the equation $4^{x}=1+2.3^{x-\frac{1}{2}}$ has a unique
real solution, so equal to $\beta _{4}$. Now we claim that 
\begin{equation}\label{3.21}
b_{4}\leq \beta _{4}\mbox{.}  %\tag{3.21}
\end{equation}%
Indeed, if $z=x+iy$ a zero of $G_{4}(z)$, one has $1+2^{z}+4^{z}=-3^{z}$.
Then the real and imaginary parts are equal, 
\[
1+2^{x}\cos (y\log 2)+4^{x}\cos (y\log 4)=-3^{x}\cos (y\log 3)\mbox{,}
\]%
\[
2^{x}\sin (y\log 2)+4^{x}\sin (y\log 4)=-3^{x}\sin (y\log 3)\mbox{.}
\]%
By squaring and adding, we obtain the same expression that in (3.9), but the
right side term which is equal to $3^{2x}$. That means that $x=\Re z$, $%
y=\Im z$ satisfy equation (3.16), with $x_{0}$ replaced by $x$, so
necessarily $4^{x}\leq 1+2.3^{x-\frac{1}{2}}$. Therefore (3.21) follows, as
claimed. In consequence, if $x_{0}\in \left[ 0,b_{4}\right) $, by (3.21) one
has $x_{0}<\beta _{4}$ which implies that $4^{x_{0}}<1+2.3^{x_{0}-\frac{1}{2}%
}$. That is, $x_{0}$ satisfies (3.20) and then we have 
\begin{equation}\label{3.22}
x_{0}<x_{1}\leq b_{4,x_{0}}\mbox{, for all }x_{0}\in \left[ 0,b_{4}\right) 
\mbox{.}  %\tag{3.22}
\end{equation}%
Let $x_{0}$ be an arbitrary point such that $0<x_{0}<b_{4}$. The point $%
b_{4,x_{0}}^{-}$, by Proposition 8, is obtained by making $y=0$ in equation
(3.16), so $b_{4,x_{0}}^{-}$ is the real solution of the equation 
\begin{equation}\label{3.23}
1+2^{x}+4^{x}=3^{x_{0}}\mbox{.}  %\tag{3.23}
\end{equation}%
Then we claim that $b_{4,x_{0}}^{-}<x_{0}$. Indeed,\ if $b_{4,x_{0}}^{-}\geq
x_{0}$, from (3.23), we would have 
\[
3^{x_{0}}=1+2^{b_{4,x_{0}}^{-}}+4^{b_{n,x_{0}}^{-}}\geq
1+2^{x_{0}}+4^{x_{0}}>3^{x_{0}}\mbox{,}
\]%
a contradiction. Then, it follows, according to (3.22) that $%
b_{4,x_{0}}^{-}<x_{0}<b_{4,x_{0}}$, which means that the line $x=x_{0}$
meets to the level curve $\left\vert G_{4}^{\ast }(z)\right\vert =3^{x_{0}}$%
, and, by Theorem 4, $x_{0}\in R_{4}$. As consequence, $\left(
0,b_{4}\right) \subset $ $R_{4}$. Now, noting that $R_{4}$ is closed, one
deduces that $\left[ 0,b_{4}\right] \subset R_{4}$. Finally, taking into
account (3.15), one has $\left[ a_{4},b_{4}\right] \subset R_{4}$ and since $%
R_{4}\subset \left[ a_{4},b_{4}\right] $, we definitively get $\left[
a_{4},b_{4}\right] =R_{4}$. Now, by using (2.6), the set $R^{(4)}$
corresponding to the partial sum $\zeta _{4}(z)$ also satisfies 
\[
R^{(4)}=\left[ -b_{4},-a_{4}\right] =\left[ a^{(4)},b^{(4)}\right] \mbox{.}
\]%
The proof is now completed.
\end{proof}

\section{Distribution of the zeros of $G_{n}^{\ast }(z)$ on the half-planes $%
\Re z>0$ and $\Re z<0$}

Now it is needed to prove a similar result to that of Theorem 2 for the
functions $G_{n}^{\ast }(z)$.

\begin{theorem}
Every function $G_{n}^{\ast }(z)$, $n>2$, possesses infinitely many zeros in
each half-plane $\left\{ z:\Re z<0\right\} $ and $\left\{ z:\Re z>0\right\} $
except for $n=3$, $4$, whose zeros are all imaginary.
\end{theorem}

\begin{proof}
For $n=3$ one has $p_{k_{3}}=3$. Then $G_{3}^{\ast }(z)=1+2^{z}$, whose
zeros are 
\[
\frac{i\pi (2k+1)}{\log 2}\mbox{, }k\in 
%TCIMACRO{\U{2124} }%
%BeginExpansion
\mathbb{Z}
%EndExpansion
\mbox{.} 
\]%
For $n=4$, $p_{k_{4}}=3$, so $G_{4}^{\ast }(z)=1+2^{z}+4^{z}$ and then its
zeros are 
\[
\frac{i2\pi (3k+1)}{3\log 2}\mbox{, }\frac{i2\pi (3k+2)}{3\log 2}\mbox{, }%
k\in 
%TCIMACRO{\U{2124} }%
%BeginExpansion
\mathbb{Z}
%EndExpansion
\mbox{.} 
\]%
Hence, for $n=3$, $4$, the zeros of $G_{n}^{\ast }(z)$ are all imaginary
and, consequently, the theorem follows in both cases.

Assume that, for some $m>4$, there exists a function $G_{m}^{\ast }(z)$
whose zeros $(\alpha _{m,l})_{l=1,2,...}$ are of the form \ 
\begin{equation}\label{4.1}
\Re \alpha _{m,l}\neq 0\mbox{ if }l\leq p\mbox{, }\Re \alpha _{m,l}=0\mbox{
if }l>p\mbox{,}  %\tag{4.1}
\end{equation}%
for some positive integer $p$. By noting that $G_{m}^{\ast }(\overline{z})=%
\overline{G_{m}^{\ast }(z)}$ for all $z\in 
%TCIMACRO{\U{2102} }%
%BeginExpansion
\mathbb{C}
%EndExpansion
$, the zeros of $G_{m}^{\ast }(z)$ are conjugate and then necessarily $p$ is
even. Hence, for some integer $q$, the number $p=2q\geq 2$ and the zeros $%
\alpha _{m,l}$, if $l>p$, are of the form $\pm iy_{j}$ with $y_{j}>0$ for
all $j\geq 1$. Let us denote by $P(z)$ the polynomial with leader
coefficient $1$ and zeros $(\alpha _{m,l})_{l=1,2,...p}$. Then%
\begin{equation}\label{4.2}
P(z)=(z^{2}-2a_{1}z+\left\vert \alpha _{m,1}\right\vert ^{2})\cdot \cdot
\cdot (z^{2}-2a_{q}z+\left\vert \alpha _{m,q}\right\vert ^{2})\mbox{,} 
%\tag{4.2}
\end{equation}%
where $a_{l}:=\Re \alpha _{m,l}$, $l=1,...,q$. By considering the function 
\[
H_{m}^{\ast }(z):=\frac{G_{m}^{\ast }(z)}{P(z)}\mbox{,} 
\]%
entire of order $1$ with zeros $\alpha _{m,l}$ for $l>p$, by Hadamard's
factorization theorem $\left[ 1\mbox{, Th. 4.4.3}\right] $, we have 
\begin{equation}\label{4.3}
H_{m}^{\ast }(z)=e^{A_{m}z+B_{m}}\prod_{j=1}^{\infty }\left( 1+\frac{z^{2}}{%
y_{j}^{2}}\right) \mbox{.}  %\tag{4.3}
\end{equation}%
The constants $B_{m}$ and $A_{m}$ are given by the relations 
\[
e^{B_{m}}=H_{m}^{\ast }(0)=\frac{m-1}{\left\vert \alpha _{m,1}...\alpha
_{m,q}\right\vert ^{2}}\mbox{,} 
\]%
\[
A_{m}=\frac{(H_{m}^{\ast })^{\prime }(0)}{e^{B}}=\frac{\log (\frac{m!}{%
p_{k_{m}}})}{m-1}+2\left( \frac{a_{1}}{\left\vert \alpha _{m,1}\right\vert
^{2}}+...+\frac{a_{q}}{\left\vert \alpha _{m,q}\right\vert ^{2}}\right) 
\mbox{,} 
\]%
where $p_{k_{m}}$ denotes, as usual, the last prime not exceeding $m$.
Noticing (4.3), $H_{m}^{\ast }(z)e^{-A_{m}z}$ is an even function, then 
\[
H_{m}^{\ast }(z)e^{-A_{m}z}=H_{m}^{\ast }(-z)e^{A_{m}z}\mbox{ for all }z\in 
%TCIMACRO{\U{2102} }%
%BeginExpansion
\mathbb{C}
%EndExpansion
\mbox{.} 
\]%
Hence, substituting $H_{m}^{\ast }(z)$ by $\frac{G_{m}^{\ast }(z)}{P(z)}$,
we get 
\begin{equation}\label{4.4}
\frac{P(-z)}{P(z)G_{m}^{\ast }(-z)}=\frac{e^{2A_{m}z}}{G_{m}^{\ast }(z)}%
\mbox{, for all }z\in 
%TCIMACRO{\U{2102} }%
%BeginExpansion
\mathbb{C}
%EndExpansion
\mbox{.}  %\tag{4.4}
\end{equation}%
Now, if we assume that any zero of $P(z)$ is a zero of $P(-z)$, then $%
(\alpha _{m,l})_{l=1,2,...p}$ are conjugate and opposite, so $q$ is even,
and $P(z)$ coincides with $P(-z)$. Thus $\frac{a_{1}}{\left\vert \alpha
_{m,1}\right\vert ^{2}}+...+\frac{a_{q}}{\left\vert \alpha
_{m,q}\right\vert ^{2}}=0$, so%
\begin{equation}\label{4.5}
A_{m}=\frac{\log (\frac{m!}{p_{k_{m}}})}{m-1}\mbox{.}  %\tag{4.5}
\end{equation}%
By writing $G_{m}^{\ast }(z)=1+2^{z}+...+\left( m^{\ast }\right) ^{z}$ (it
does not have the term $p_{k_{m}}^{z}$), where%
\[
m^{\ast }:=\left\{ 
\begin{array}{c}
m-1\mbox{, if }m\mbox{ is prime} \\ 
m\mbox{, otherwise.}%
\end{array}%
\right. \mbox{,} 
\]%
we claim that $2A_{m}-\log m^{\ast }>0$ for all $m>4$. Indeed, by assuming
that $m$ is prime, $m=p_{k_{m}}$, so $m^{\ast }=m-1$ and, by (4.5), one has 
\begin{equation}\label{4.6}
2A_{m}-\log m^{\ast }=2\frac{\log ((m-1)!)}{m-1}-\log (m-1)\mbox{.} 
%\tag{4.6}
\end{equation}%
But, $2\frac{\log ((m-1)!)}{m-1}-\log (m-1)$ is positive by virtue of the
inequality 
\begin{equation}\label{4.7}
(k!)^{2}>k^{k}\mbox{ for all integer }k>2\mbox{,}  %\tag{4.7}
\end{equation}%
which can be easily proved by induction and by the mean value theorem
applied on the $\log $ function. Hence the claim follows provided that $m$
be a prime number. Let us suppose that $m$ is a composite number greater
than $5$. From (4.5) and the definition of $m^{\ast }$, we get 
\[
2A_{m}-\log m^{\ast }=2\frac{\log (\frac{m!}{p_{k_{m}}})}{m-1}-\log m\mbox{.%
} 
\]%
Then $2A_{m}-\log m^{\ast }$ is positive if and only if $%
(m!)^{2}>m^{m-1}(p_{k_{m}})^{2}$. But this last inequality, taking into
account that $p_{k_{m}}\leq m-1$, can be immediately deduced from the
inequality, sharpest than (4.7), 
\[
(k!)^{2}>k^{k-1}(k-1)^{2}\mbox{ for all integer }k>5\mbox{,} 
\]%
which can be also demonstrated by using induction and by the mean value
theorem applied on the $\log $ function. Hence, 
\begin{equation}\label{4.8}
2A_{m}-\log m^{\ast }>0\mbox{ for any }m>4\mbox{,}  %\tag{4.8}
\end{equation}%
as we claimed. Now, by making $z=$ $x>0$ in (4.4) and by taking the limit
when $x\rightarrow +\infty $, because of (4.8), the left side of (4.4) tends
to $1$ whereas the right side tends to $+\infty $, a contradiction.
Therefore, there is at least\ one zero of $P(z)$ which is not a zero of $%
P(-z)$. Then, by writing (4.4) of the form 
\begin{equation}\label{4.9}
\frac{P(-z)}{P(z)}=\frac{e^{2A_{m}z}G_{m}^{\ast }(-z)}{G_{m}^{\ast }(z)}%
\mbox{, for all }z\in 
%TCIMACRO{\U{2102} }%
%BeginExpansion
\mathbb{C}
%EndExpansion
\mbox{,}  %\tag{4.9}
\end{equation}%
the left side of (4.9) is a meromorphic function with at least a pole,
whereas the right side is a quotient of exponential polynomials with a
finite number of poles. It means, by Shields's theorem $\left[ 21\right] $,
that the right side of (4.9) is an exponential polynomial and then we are
led to a new contradiction. Therefore, there is no integer $m>4$ for which
the function $G_{m}^{\ast }(z)$ have its zeros of the form (4.1) and, in
particular, are not all of them situated on the imaginary axis. Indeed, if
this were so, by defining the polynomial $P(z)$ as identically $1$, we would
be led to the following contradiction: by taking the limit when $%
z=x\rightarrow +\infty $ in (4.4), the left side of (4.4) tends to $1$
whereas the right side one tends to $+\infty $. Consequently, all the
functions $G_{n}^{\ast }(z)$, $n>4$, have infinitely many zeros in at least
one of the half-planes $\left\{ z:\Re z<0\right\} $, $\left\{ z:\Re
z>0\right\} $. Let $(z_{n,l}^{\ast })_{l=1,2,...}$ be all the zeros of $%
G_{n}^{\ast }(z)$ for a fixed $n>4$. Suppose, for example, that $\Re
z_{n,l}^{\ast }\geq 0$ for all $l\geq 1$ and let $z_{n,j}^{\ast }$ be a zero
of $G_{n}^{\ast }(z)$ such that $\Re z_{n,j}^{\ast }>0$. Then if we take $%
\epsilon =\Re z_{n,j}^{\ast }$, noticing $G_{n}^{\ast }(z)$ is an
almost-periodic function, because Lemma 1, $G_{n}^{\ast }(z)$ has infinitely
many zeros in the strip 
\[
S_{\frac{\epsilon }{2}}:=\left\{ z:\frac{\epsilon }{2}<\Re z<\frac{3\epsilon 
}{2}\right\} \mbox{,} 
\]%
which implies that 
\begin{equation}\label{4.10}
\sum_{l=1}^{\infty }\Re z_{n,j}^{\ast }=+\infty \mbox{.}  %\tag{4.10}
\end{equation}%
On the other hand, since $G_{n}^{\ast }(z)$ is an exponential polynomial
with all its coefficients equal to $1$, by Ritt's formula $\left[ 19,\mbox{
formula (9)}\right] $, we have 
\[
\sum_{l=1}^{\infty }\Re z_{n,l}^{\ast }=O(1)\mbox{,} 
\]%
contradicting (4.10). Then $G_{n}^{\ast }(z)$ has at least a zero in the
half-plane $\left\{ z:\Re z<0\right\} $ and again, by Lemma 1, $G_{n}^{\ast
}(z)$ has infinitely many zeros with negative real part. This completes the
proof.
\end{proof}

The next result is obtained directly from the above theorem.

\begin{corollary}
The bounds corresponding to the functions $G_{n}^{\ast }(z)$ 
\[
a_{n}^{\ast }:=\inf \left\{ \Re z:G_{n}^{\ast }(z)=0\right\} \mbox{, }%
b_{n}^{\ast }:=\sup \left\{ \Re z:G_{n}^{\ast }(z)=0\right\} 
\]
satisfy 
\[
a_{n}^{\ast }=b_{n}^{\ast }=0\mbox{ if }n=3\mbox{, }4 
\]%
and 
\[
a_{n}^{\ast }<0\mbox{, }b_{n}^{\ast }>0\mbox{ for all }n>4\mbox{.} 
\]
\end{corollary}

Now we are ready to prove the main result of this paper.

\begin{theorem}
For each $n>2$, there exists $\delta _{n}>0$ such that $\left[ a_{n,}\delta
_{n}\right] \subset R_{n}$ and $\left[ -\delta _{n},b^{(n)}\right] \subset
R^{(n)}$.
\end{theorem}

\begin{proof}
From Corollary 3 we have $a_{n}<0$ and $b_{n}>0$, for all $n>2$. By Lemma
13, $b_{n}^{\ast }<b_{n,x_{0}}$ for every integer $n>2$ and for arbitrary
real $x_{0}$. Then, in particular, for $x_{0}=a_{n}$, we have 
\begin{equation}\label{4.11}
b_{n}^{\ast }<b_{n,a_{n}}\mbox{ for every }n>2\mbox{.}  %\tag{4.11}
\end{equation}%
Since $a_{n}\in R_{n}$, by Theorem 10, 
\begin{equation}\label{4.12}
\left[ a_{n},b_{n,a_{n}}\right] \cap \left[ a_{n,}b_{n}\right] \subset R_{n}%
\mbox{.}  %\tag{4.12}
\end{equation}%
Because of Corollary 16, $0\leq b_{n}^{\ast }$ for every $n>2$. Then by
(4.11), $0<b_{n,a_{n}}$ for every $n>2$. Hence, 
\[
\delta _{n}:=\min \left\{ b_{n,a_{n}},b_{n}\right\} 
\]%
is a positive number for every $n>2$. Thus, by (4.12), we get 
\[
\left[ a_{n},b_{n,a_{n}}\right] \cap \left[ a_{n,}b_{n}\right] =\left[
a_{n,}\delta _{n}\right] \subset R_{n}\mbox{ for every }n>2\mbox{.} 
\]%
Finally, from (2.6), we obtain 
\[
\left[ -\delta _{n},-a_{n}\right] =\left[ -\delta _{n},b^{(n)}\right]
\subset R^{(n)}\mbox{ for every }n>2\mbox{,} 
\]%
and then the theorem follows.
\end{proof}

\begin{remark}
Observe that the proof of Theorem 17, valid for all $n>2$, is essentially
based on the old notion of level curve. Nevertheless (1.7) follows for all $%
n>$ $N$, for some positive integer $N$ whose existence depends of the Prime
Number Theorem (see $\left[ 15\mbox{, Lemma }7\right] $). Then $\left[ 15%
\mbox{, Th.12}\right] $ improves the Theorem 17 of an optimal form, provided
that $n>N$, because the $\delta _{n}$'s are equal to $b_{n}=-a^{(n)}=n\log
2+o(n)$, $n\rightarrow \infty $, by virtue of (1.5). The same occurs when $%
n=3$, $4$, by virtue of Corollary 11 and Theorem 14, respectively.
\end{remark}

An important conclusion is deduced from Theorem 17, namely, each set $\cap
_{n=3}^{\infty }R_{n}$, $\cap _{n=3}^{\infty }R^{(n)}$ contains a
non-degenerated closed interval.

\begin{corollary}
There exist $\delta \geq 0$ and $a<0$ such that $\cap _{n=3}^{\infty
}R_{n}\supset \left[ a,\delta \right] $ and $\cap _{n=3}^{\infty
}R^{(n)}\supset \left[ -\delta ,-a\right] $.
\end{corollary}

\begin{proof}
By virtue of Theorem 17, we define $\delta :=\inf \left\{ \delta
_{n}:n>2\right\} $ and then, since $\delta _{n}>0$ for all $n>2$,
necessarily $\delta \geq 0$. By (1.4), all the terms of the sequence $\left(
b^{(n)}\right) _{n>2}$, except at most a finite quantity of them, are
arbitrarily close to $1$. By Corollary 3, $b^{(n)}>0$ for all $n>2$, then $%
b:=\min \left\{ b^{(n)}:n>2\right\} $ is a positive number. On the other
hand, by (2.6), one has $-a_{n}=b^{(n)}$ for all $n>2$, so $a:=\max \left\{
a_{n}:n>2\right\} =-b<0$. Now the proof is completed.
\end{proof}

\begin{corollary}
For all $n>2$, there exist infinitely many zeros, having positive and
negative real part, of $G_{n}(z)$ and $\zeta _{n}(z)$ arbitrarily close to
the imaginary axis. In particular, the point $0$ is an accumulation point
common to the sets $P_{G_{n}}=\left\{ \Re z:G_{n}(z)=0\right\} $ and $%
P_{\zeta _{n}}=\left\{ \Re z:\zeta _{n}(z)=0\right\} $ for all $n>2$.
\end{corollary}

\begin{proof}
From Theorem 17, it is enough to take into account that $0$ is an interior
point to each set $R_{n}:=\overline{P_{G_{n}}}$ , $R^{(n)}:=\overline{%
P_{\zeta _{n}}}$, for all $n>2$.
\end{proof}

In $\left[ 20\mbox{, Conjecture }15\right] $ Sepulcre and Vidal conjectured
that the point $0\in R_{n}$ for all $n\geq 2$. Now, as an easy consequence
of Theorem 17, we can give a positive answer to such question.

\begin{corollary}
The conjecture of Sepulcre and Vidal is true.
\end{corollary}

\begin{proof}
By Corollary 19, $\cap _{n=3}^{\infty }R_{n}\supset \left[ a,\delta \right] $%
, where $a<0$ and $\delta \geq 0$. Then, in particular, $0\in R_{n}$ for all 
$n>2$. Noting (1.6), the set $R_{2}=\left\{ 0\right\} $, therefore it
follows that $0\in R_{n}$ for all $n\geq 2$. Consequently Sepulcre and
Vidal's conjecture is true.
\end{proof}

\end{document}